\documentclass[11pt,reqno]{amsart}
\usepackage[utf8]{inputenc}
\usepackage{amsmath,amssymb}
\usepackage{hyperref}
\usepackage{mathtools}
\usepackage{cite}
\usepackage{tikz}
\usetikzlibrary{automata}
\usepackage{amsmath}
\usepackage{mathrsfs}
\usepackage{centernot}
\numberwithin{equation}{section}
\newtheorem{thm}{Theorem}[section]
\newtheorem{cor}[thm]{Corollary}

\newtheorem{prop}[thm]{Proposition}
{\theoremstyle{remark}
\newtheorem{rmk}[thm]{Remark}}
{\theoremstyle{remark}
\newtheorem{example}[thm]{Example}}

{\theoremstyle{definition}

\newcommand{\malcev}{\mathbin{\hbox{$\bigcirc$\rlap{\kern-8.25pt\raise0,50pt\hbox{${\tt
  m}$}}}}}
\newcommand{\smalcev}{\mathbin{\hbox{$\bigcirc$\rlap{\kern-7pt\raise0,30pt\hbox{${\tt
  m}$}}}}}

\newcommand{\J}{\mathrel{\mathscr J}} % J - relation
\newcommand{\R}{\mathrel{\mathscr R}} % R - relation
\newcommand{\eL}{\mathrel{\mathscr L}} % L - relation

\newcommand{\Ext}{\mathop{\mathrm{Ext}}\nolimits}
\newcommand{\Hom}{\mathop{\mathrm{Hom}}\nolimits}
\newcommand{\Aff}{\mathop{\mathrm{Aff}}\nolimits}
\newcommand{\St}{\mathop{\mathrm{St}}\nolimits}
\newcommand{\Ind}{\mathop{\mathrm{Ind}}\nolimits}

\newcommand{\Res}{\mathop{\mathrm{Res}}\nolimits}

\begin{document}
\title{The quiver of an affine monoid}
% \author{M.H. Shahzamanian and B. Steinberg}
\author{M.H. Shahzamanian}
\author{B. Steinberg}
\address{M.H. Shahzamanian\\ Centro de Matem\'atica e Departamento de Matem\'atica, Faculdade de Ci\^{e}ncias,
Universidade do Porto, Rua do Campo Alegre, 687, 4169-007 Porto,
Portugal}
\email{m.h.shahzamanian@fc.up.pt}
\address{B. Steinberg  \\ Department of Mathematics\\
    City College of New York\\
    Convent Avenue at 138th Street\\
    New York, New York 10031\\
    USA.}
\email{bsteinberg@ccny.cuny.edu}
\subjclass[2010]{20M30,20M25}
\keywords{monoid representation theory, %graphs, posets,
simple modules, affine monoids, quivers of algebras}

\begin{abstract}
If $R$ is a finite commutative ring, then the affine monoid of $R$ is the monoid of all affine mappings $x\mapsto ax+b$ on $R$.  Alternatively, it is the semidirect product of the multiplicative monoid of $R$ with the additive group of $R$.  In this paper we compute the Gabriel quiver of the complex algebra of the affine monoid of any finite commutative ring.
\end{abstract}
\maketitle
%\tableofcontents

\section{Introduction}
The representation theory of monoids provides an interesting bridge between group representation theory and the representation theory of finite dimensional algebras.  On the one hand, monoid algebras are seldom semisimple and so questions about their quivers, indecomposable modules and homological properties are interesting.  On the other hand, their irreducible representations are constructed from irreducible representations of finite groups and so it is often possible to apply group representation theory to answer questions about monoid representation theory; see for instance~\cite{Ben-Rep-Monoids-2016,DO,Mar-Ste-Proj-2018}.  Monoid representation theory also has applications to combinatorics~\cite{MSS,ourmemoirs} and probability theory as it can be used to compute eigenvalues for random walks coming from monoids~\cite{Brown1,ayyer_schilling_steinberg_thiery.2013,Ben-Rep-Monoids-2016}.

In a recent paper~\cite{Ben-Ayy}, Ayyer and the second author studied Markov chains coming from random applications of $ax+b$ mappings on a finite commutative ring $R$ using the representation theory of the affine monoid $\Aff(R)$ of all such maps.  Note that in this paper all rings are assumed to be unital.   Using this theory, they were able to compute the spectrum of the transition matrix of the random walk under certain restrictions on the underlying probability measure.  Finite group representation theory, in particular, induced representations, Mackey theory and Frobenius reciprocity, played a key role.  It is therefore natural to consider further the representation theory of the complex algebra $\mathbb C\Aff(R)$.  This is not a semisimple algebra for any non-zero ring $R$.  In fact, if $F$ is a finite field it can be deduced from the results of~\cite[Section~6.2]{MarSte11} that $\mathbb C\Aff(F)$ is a hereditary algebra that is Morita equivalent to the direct product of the ring of $2\times 2$ upper triangular matrices over $\mathbb C$ with $|F|-1$ copies of $\mathbb C$.

In this paper we compute the quiver of $\mathbb C\Aff(R)$ for any finite commutative ring $R$.  It was observed in~\cite{Ben-Ayy} that the idempotents of $\Aff(R)$ form a left regular band, that is, a monoid of idempotents in which each principal right ideal has a unique generator.  Such monoids belongs to a class of monoids whose quivers are described in a paper of Margolis and the second author~\cite{Mar-Ste-Proj-2018}.  But this description is more theoretical than practical.  To each pair $V$ and $W$ of simple modules of the monoid, they associate finite groups $G_V$ and $G_W$ and simple modules $\widetilde V$ and $\widetilde W$ for these groups such that the number of arrows from $V$ to $W$ is the multiplicity of the simple $\mathbb C[G_W\times G_V]$-module $\widetilde W\otimes_{\mathbb C}\mathrm{Hom}_{\mathbb C}(\widetilde V,\mathbb C)$ in a certain $\mathbb C[G_W\times G_V]$-module constructed as a submodule of a permutation module built from monoid data.  Our goal in this paper is to work this out for $\mathbb C\Aff(R)$ in order to get a concrete description of the quiver.

The basic strategy is to observe that a finite commutative ring is a direct product of local rings and the affine monoid is a direct product of the corresponding affine monoids.  Since the algebra of a direct product of monoids is a tensor product of their respective algebras and there is a well-known method to compute the quiver of a tensor product from the quivers of the factors~\cite{QuiverTensor}, we can reduce to the case of a local ring $R$.  For a local ring $R$, the only groups that arise as $G_V$ and $G_W$ are the trivial group and the affine group of all $ax+b$ mappings with $a$ a unit.  This latter group is a semidirect product of abelian groups and so its representation theory can be understood using the Mackey machine.

It turns out that if $R$ is a principal ideal ring, then there are no multiple edges in the quiver.  In general, there can be multiple edges.  We remark that if $R$ is reduced (or equivalently semisimple), then $\mathbb C\Aff(R)$ is Morita equivalent to a tensor product of path algebras of acyclic quivers and so a full quiver presentation can be obtained using the methods of~\cite{QuiverTensor}.  It would be interesting to find a quiver presentation in general.

The paper is organized as follows. We begin with a preliminary section on monoids, finite dimensional algebras and the representation theory of monoids.  We then turn to the computation of the quiver and explicitly draw the quivers of $\mathbb C\Aff(\mathbb Z_4)$, $\mathbb C\Aff(\mathbb Z_6)$, $\mathbb C\Aff(\mathbb Z_8)$ and $\mathbb C\Aff(\mathbb Z_9)$ to illustrate the method.

%%%%%%%%%%%%%%%%%%%%%%%%%%%%%%%%%%%%%%%%%%%%%%%%%%%%%%%%%%%%%%%%%%%%%%%%%%%%%%%%%%%%%%%%%%%%%%%%%%%%%%%%%%%%%%%%%%%%%%%%%%%%%%%%%%%%%%%%%%%%%%%%%%%%%%%%%%%%%%%%%%%%%%%%%%%%%%%%%%%%%%%%%%%%%%%%%%%%%%%%%%%%%%%%%%%%%%%%%%%%%%%%%%%%%%%%%%%%%%%%%%%%%%%%%%%%%%%%%%%%%%%%%%%%%%%%%%%%%%%%%%%%%%%%%%%%%%%%%%%%%%%%%%%%%%%%%%%%%%%%%%%%%%%%%%%%%%%%%%%%%%%%%%%%%%%%%%%%%%%%%%%%%%%%%%%%%%%%%%%%%%%%%%%%%%%%%%%%%%%%%%%%%%%%%%%%%%%%%%%%%%%%%%

\section{Preliminaries}

\subsection{Monoids}
For standard notation and terminology relating to monoids, we refer the reader to~\cite{Alm,Cli-Pre,Rho-Ste}.
Let $M$ be a finite monoid. Let $a,b\in M$. We say that $a\R b$ if $aM = bM$, $a\eL b$ if $Ma = Mb$ and  $a\J b$ if $MaM = MbM$.
The relations $\R$, $\eL$ and $\J$ are equivalence relations first introduced by Green~\cite{Gre} and are known as Green's relations.
We denote by $R_a$, $L_a$ and $J_a$, respectively, the $\R$, $\eL$ and $\J$-class containing $a$.
%An important property of finite monoids is the stability property that $J_m\cap Mm = L_m$ and $J_m\cap mM = R_m$, for every $m \in M$ \cite[Definition A.2.1]{Rho-Ste}.
%For $\J$-classes $J_a$ and $J_b$, we can define the partial order $\leq$ as follows:
%$$MaM\subseteq MbM\ \text{if and only if}\ J_a\leq J_b.$$

An element $e$ of $M$ is an \emph{idempotent} if $e^2 = e$. The set of all idempotents of $M$ is denoted by $E(M)$.
%more generally, for any $X\subseteq M$, we put $E(X)=X\cap E(M)$.
An idempotent $e$ of $M$ is the identity of the monoid $eMe$. The group of units $G_e$ of $eMe$ is called the \emph{maximal subgroup} of $M$ at $e$.

An element $m$ of $M$ is (von Neumann) \emph{regular} if there exists an element $n\in M$ such that $mnm=m$. Note that an element $m$ is regular if and only if $m\eL e$, for some $e\in E(M)$, if and only if $m\R f$, for some $f\in E(M)$. A $\J$-class $J$ of a finite monoid is \emph{regular} if all its elements are regular or, equivalently, if  $J$ contains an idempotent.
%Note that if $N$ is a submonoid of $M$ and $a,b\in N$ are regular in $N$, then $a\K b$ in $N$ if and only if $a\K b$ in $M$
%where $\K$ is any of $\R, \eL$ or $\HH$ \cite[Proposition A.1.16]{Rho-Ste}.

Let $e \in E(M)$. The monoid $M$ acts on the right of the $\R$-class $R_e$ by partial mappings.
The definition is as follows: for $m\in M$ and $r \in R_e$,
   \[r\cdot m = \begin{cases}
      rm, & \text{if}\ rm\in R_e\\
       \text{undefined}, & \text{else.}\end{cases}\]
The monoid $M$ is \emph{$\R$-trivial} if $mM=nM$ implies $m = n$.
If the idempotents of $M$ generate an $\mathscr R$-trivial monoid, then the above action is by partial injections (cf.~\cite[Theorem~4.8.3]{Rho-Ste}).

The group of units of a monoid $M$ will be denoted by $U(M)$; we use the same notation for the group of units of a ring.

Let $M$ be a monoid.
The equivalence relation $\widetilde{\eL}$ is defined on $M$ by $m \mathrel{\widetilde{\eL}} n$ if, for all idempotents $e \in E(M)$, we have $me = m$ if and only if $ne = n$~\cite{Fou-Gom-Gou}. Note that ${\eL}\subseteq {\widetilde{\eL}}$.  The $\widetilde{\eL}$-class of $m$ is denoted $\widetilde{L}_m$.
We put for $n\in M$, $I(n) = \{m\in M\mid n\notin MmM\}$.  Note that $I(n)$ is a two-sided ideal.

\subsection{Finite dimensional algebras}

%A finite sequence of module homomorphisms, $$A_0 \xrightarrow{f_1} A_1 \xrightarrow{f_2}  \ldots \xrightarrow{f_{n-1}} A_{n-1} \xrightarrow{f_n} A_n,$$ with $\im f_i = \ker f_{i+1}$, for $i = 1,2, \ldots, n-1$, is called an exact sequence and an exact sequence of the form $$0 \rightarrow A \xrightarrow{f} B \xrightarrow{g} C \rightarrow 0$$ is called a short exact sequence. We
%refer the reader for more details to~\cite{Hun}.
Fix an algebraically closed field $\mathbb K$.  The reader is referred to~\cite{Ibr} for background on the representation theory of finite dimensional algebras.  In this paper, modules are by default left modules.
  If $A$ is a finite dimensional algebra over $\mathbb{K}$, then the (Gabriel) \emph{quiver} of $A$ is the directed graph $Q(A)$
with vertex set $Q(A)_0$ the set of isomorphism classes of simple $A$-modules and with edge set $Q(A)_1$ described as follows.
Let $[S]$ and $[S']$ be isomorphism classes of simple $A$-modules.
The number of arrows from $[S]$ to $[S']$ is $\dim\Ext^1_A(S,S')$.
%If the projective cover of the simple module $S$ is $P$, then one easily checks that $$\Ext^1_A(S,S')\cong \Hom_A(\rad P/\rad^2 P, S').$$
%We refer the reader for more details on quivers of algebras to~\cite[Chapter 17]{Ben-Rep-Monoids-2016} and~\cite{Mar-Ste-Proj-2018}.

Let $e_1,\ldots, e_r$ be an orthogonal set of primitive idempotents of $A$ such that each projective indecomposable $A$-module is isomorphic to exactly one $Ae_i$.  Then $e=e_1+\cdots+e_r$ is an idempotent and $eAe$ is the unique (up to isomorphism) basic algebra Morita equivalent to $A$~\cite[Chapter~I.6]{Ibr}.  Recall that $B$ is \emph{basic} if $B/\mathrm{rad}(B)\cong \mathbb K^n$ for some $n\geq 1$.  Note that $Q(A)\cong Q(eAe)$ and that $eAe$ is the quotient of the path algebra of $Q(A)$ by an admissible ideal by  a famous result of Gabriel; see~\cite[Theorem~II.3.7]{Ibr}.  Recall that a finite dimensional algebra $A$ is \emph{hereditary} if each submodule of a projective module is projective.  The algebra $A$ is hereditary if and only if its basic algebra is isomorphic to the path algebra of an acyclic quiver by a result of Gabriel~\cite[Theorem~VII.1.7]{Ibr}.

Since $\mathbb K$ is an algebraically closed field, if $A$ and $B$ are $\mathbb K$-algebras and $e_1,\ldots, e_m$ and $f_1,\ldots, f_n$ form a complete set of orthogonal primitive idempotents of $A$ and $B$, respectively, then the $e_i\otimes f_j$ form a complete set of orthogonal primitive idempotents of $A\otimes_{\mathbb K} B$ and, moreover, $(A\otimes_{\mathbb K}B)(e_i\otimes f_j)\cong (A\otimes_{\mathbb K}B)(e_k\otimes f_\ell)$ if and only if $Ae_i\cong Ae_k$ and $Bf_j\cong Bf_\ell$. This is proved in~\cite[Proposition~11.3]{FrobAlg} for the special case where $B$ is the opposite algebra of $A$, but the proof is the same in the general case.
In particular, if $e,f$ are the idempotents constructed from these sets with $eAe$ and $fBf$ the basic algebras of $A$ and $B$ respectively, then $e\otimes f$ is an idempotent with $(e\otimes f)(A\otimes_{\mathbb K} B)(e\otimes f)\cong eAe\otimes_{\mathbb K}fBf$ the basic algebra of $A\otimes_{\mathbb K} B$.
It follows that the basic algebra of $A\otimes_{\mathbb K} B$ is the tensor product of the basic algebras of $A$ and $B$.

An important fact is that if $A$ and $B$ are basic $\mathbb K$-algebras, then $Q(A\otimes_{\mathbb K} B)$ can be built from $Q(A)$ and $Q(B)$ as follows (see~\cite{QuiverTensor}): $Q(A\otimes_{\mathbb K} B)_0 = Q(A)_0\times Q(B)_0$; $Q(A\otimes_{\mathbb K}B)_1=(Q_1(A)\times Q_0(B))\cup (Q_0(A)\times Q_1(B))$ where if $e$ is an edge from $v$ to $w$ in $Q(A)$ and $z$ is a vertex of $Q(B)$, then $(e,z)$ goes from $(v,z)$ to $(w,z)$ and,  similarly, if $f$ is an edge form $r$ to $s$ in $Q(B)$ and $v$ is a vertex of $Q(A)$, then $(v,f)$ goes from $(v,r)$ to $(v,s)$.  The admissible ideal for $A\otimes_{\mathbb K} B$ (which is basic) can likewise be constructed from the admissible ideals for $A$ and for $B$~\cite{QuiverTensor}.

We put the observations of the previous two paragraphs into a single proposition.

\begin{prop}\label{p:tensor}
Let $\mathbb K$ be an algebraically closed field and $A,B$ finite dimensional $\mathbb K$-algebras. Then $Q(A\otimes_{\mathbb K}B)$ has vertex set $Q(A)_0\times Q(B)_0$ and edge set $(Q_1(A)\times Q_0(B))\cup (Q_0(A)\times Q_1(B))$ with source and terminal vertices of edges described as above.	
\end{prop}

This will be useful for us because of the following well-known fact.

\begin{prop}\label{p:direct.prod}
If $M$ and $N$ are finite monoids and $\mathbb K$ is a field, then $\mathbb K[M\times N]\cong \mathbb KM\otimes_{\mathbb K} \mathbb KN$.\end{prop}

If $G$ is a finite group, $\mathbb K$ is a field and $V$ is a $\mathbb KG$-module, then the \emph{contragredient} module is $D(V)=\mathrm{Hom}_{\mathbb K}(V,\mathbb K)$ where $(gf)(v) = f(g^{-1}v)$ for $g\in G$, $v\in V$ and $f\colon V\to \mathbb K$ linear.  Note that in the case that $\mathbb K$ is the field of complex numbers, the character of $D(V)$ is the complex conjugate of the character of $V$.
Margolis and the second author described the quiver of $\mathbb{K}M$ for a nice class of monoids as follows~\cite{Mar-Ste-Proj-2018}.

\begin{thm}\label{U-flat}
Let $M$ be a finite monoid whose idempotents generate an $\R$-trivial monoid and $\mathbb{K}$ an algebraically closed field whose characteristic does not divide the order of any maximal subgroup of $M$. Let $e_1, \ldots, e_n$ be a complete set of idempotent representatives of the regular $\J$-classes of $M$. Then the quiver of $\mathbb{K}M$ is isomorphic to the directed graph that has vertices indexed by pairs $(e_i, [V])$ where $V$ is a simple $\mathbb{K}G_{e_i}$-module and edges as follows. If $(e, [V])$ and $(f, [W])$ are two vertices, then the number of arrow from $(e, [V ])$ to $(f, [W])$ is the multiplicity of $W\otimes_{\mathbb K} D(V)$ as an irreducible constituent of the $\mathbb{K}[G_f \times G_e]$-module $U^{\flat}$ defined as follows.

Let $\equiv$ be the least equivalence relation on $fMe$ such that:
\begin{itemize}
\item[(1)] $x \equiv y$ if $x, y \in fI(f)I(e)e$;
\item[(2)] $zx \equiv zy$ if $x, y \in L_e, z \in fI(f)$ and $x, y$ act as the same partial injection on the right of $R_e$.
\end{itemize}
Let $X$ be the set of equivalence classes of elements of $fMe$ not meeting $fI(f)I(e)e$; it is naturally a $G_f \times G_e$-set via $(g,h)[m]=[gmh^{-1}]$. Let $U^{\flat}$ be the submodule of the permutation module $\mathbb{K}X$ spanned by differences $[x]_{\equiv} - [y]_{\equiv}$ with $x, y \in L_e$
such that $x$ and $y$ act as the same partial injection on the right of $R_e$ and by those $[z]_{\equiv} \in X$ such that
$z \in (fM \cap \widetilde{L}_e) \setminus L_e$.
\end{thm}

The simple $\mathbb KM$-module associated to $(e,[V])$ can be defined as the simple top of the indecomposable module $\mathbb KL_e\otimes_{\mathbb KG_e}V$ where $G_e$ acts on the right of $L_e$ by multiplication and $M$ acts on the left of $\mathbb KL_e$ via \[m\cdot x = \begin{cases} mx, &\text{if}\ mx\in L_e\\ 0, &\text{else}\end{cases}\] for $m\in M$ and $x\in L_e$.  See~\cite[Chapter~5]{Ben-Rep-Monoids-2016} for details.

%%%%%%%%%%%%%%%%%%%%%%%%%%%%%%%%%%%%%%%%%%%%%%%%%%%%%%%%%%%%%%%%%%%%%%%%%%%%%%%%%%%%%%%%%%%%%%%%%%%%%%%%%%%%%%%%%%%%%%%%%%%%%%%%%%%%%%%%%%%%%%

\subsection{Affine monoids}
Let $R$ be a finite commutative ring.
The set of all mappings on $R$ of the form $x \mapsto ax + b$ with $a,b \in R$ with composition as the binary
operation constitutes the affine monoid $\Aff(R)$. So the product of $ax + b$ and $cx + d$ is $acx + ad + b$. Note that
$\Aff(R)$ is the semidirect product $R\rtimes M(R)$ of the multiplicative monoid $M(R)$ of $R$ with the additive group $R$, where $M(R)$ acts on $R$ by multiplication.

Note that if $R\cong R_1\times R_2$, then $\Aff(R_1\times R_2)\cong \Aff(R_1)\times \Aff(R_2)$ via the map $(a,c)x+(b,d)\mapsto (ax+b,cx+d)$.  Every finite commutative ring $R$ is isomorphic to a finite direct product of local rings $R\cong R_1\times \cdots\times R_n$~\cite{finitecommrings}.  Then $\Aff(R)\cong \Aff(R_1)\times\cdots\times \Aff(R_n)$ and hence \[\mathbb K\Aff(R)\cong \mathbb K\Aff(R_1)\otimes_{\mathbb K}\cdots\otimes_{\mathbb K}\mathbb K\Aff(R_n)\] by Proposition~\ref{p:direct.prod}.  Thus computing the quiver of the algebra of the affine monoid will reduce to the local case by Proposition~\ref{p:tensor}.

%A finite commutative ring $R$ is a \emph{principal ideal ring} if each ideal is principal.  For example, the ring $\mathbb Z_n$ of integers modulo $n$ is a principal ideal ring.  If we write $R\cong R_1\times\cdots\times R_n$  with the $R_i$ local rings, then $R$ is a principal ideal ring if and only if each $R_i$ is a principal ideal ring.

We remark that if $R$  is a finite commutative local ring with principal maximal ideal $\mathfrak m=(p)$, then there exists a minimum $n\geq 0$ with $0=\mathfrak m^n=(p^n)$.  In this case, each element of $R$ can be written in the form $a=up^k$ with $u$ a unit and $0\leq k\leq n$.  The integer $k$ is uniquely determined as the largest $k$ with $a\in \mathfrak m^k$.  Note that the annihilator of $p$ is $(p^{n-1})$.

Ayyer and the second author described the simple $\mathbb{C}\Aff(R)$-modules~\cite{Ben-Ayy}. Since we only need the case of a local ring $R$, we specialize to this case.  There is the trivial module $\mathbb C$ (where each element of $\Aff(R)$ acts via the identity) and each simple $\mathbb CU(\Aff(R))$-module $V$ extends to a simple $\mathbb C\Aff(R)$-module by having all elements of $\Aff(R)\setminus U(\Aff(R))$ annihilate $V$.  In other words, if $a\in \mathfrak m$, then $ax+b$ annihilates $V$ where $\mathfrak m$ is the maximal ideal of $R$.

As $U(\Aff(R))=R\rtimes U(R)$, the Mackey machine constructs all the irreducible representations of $U(\Aff(R))$.  We recall the construction.  See~\cite[Proposition~25]{serrerep} for details.

If $H \leq G$ is a subgroup and $V$ is a $\mathbb{C}H$-module, then the induced $\mathbb CG$-module is $\Ind^G_H V = \mathbb{C}G\otimes_{\mathbb{C}H} V$.  On the other hand, if $V$ is a $\mathbb CG$-module, then we write $\Res^G_H V$ for the restriction of $V$ to a $\mathbb CH$-module.  Note that Frobenius reciprocity says that $\mathrm{Hom}_{\mathbb CG}(\Ind_H^G V, W)\cong \mathrm{Hom}_{\mathbb CH}(V,\Res^G_H W)$ for a $\mathbb CG$-module $W$ and a $\mathbb CH$-module $V$.  We shall frequently use the Mackey decomposition theorem~\cite[Proposition~22]{serrerep}, which determines the restriction to a subgroup $K$ of a module induced from a subgroup $H$.

If $G$ acts transitively on a set $X$, then the permutation module $\mathbb CX$ is isomorphic to $\Ind_H^G 1_H$ where $H$ is the stabilizer of an element of $X$ and $1_H$ is the trivial $\mathbb CH$-module.

 If $A$ is an abelian group, let $\widehat{A}=\mathrm{Hom}(A,U(\mathbb C))$ be the group of characters of $A$, that is, the Pontryagin dual of $A$.  The left action of $U(R)$ on the additive group of $R$ by left multiplication induces a right action of $U(R)$ on $\widehat{R}$ by duality, which we can then view as a left action since $U(R)$ is commutative.  More precisely, $U(R)$ acts on $\widehat{R}$ via $(r\chi)(a)=\chi(ra)$ for $r\in U(R)$, $\chi\in \widehat{R}$ and $a\in R$.

Let $\St_{U(R)}(\chi)$ be the stabilizer of $\chi$ in $U(R)$, for $\chi\in \widehat{R}$. For $\rho\in \widehat{\St_{U(R)}(\chi)}$, there is a degree one character \[\chi\otimes\rho\colon R\rtimes \St_{U(R)}(\chi) \rightarrow U(\mathbb{C})\] given by
$(\chi\otimes\rho)(ax + b) = \chi(b)\rho(a)$ for $a \in \St_{U(R)}(\chi)$ and $b \in R$.

% Note that if $T$ is a set of coset representatives for $G/H$, then $\Ind^G_H V=\bigoplus_{t\in T} t\otimes V$ as a vector space.
%The description of the representation theory of $\mathbb{C}U(\Aff(M))$ is prepared by the following theorem (\cite[Theorem 3.7]{Ben-Ayy}).

The following result is a special case of the Mackey Machine theory, restricted to our setting; see~\cite[Proposition~25]{serrerep}.

\begin{thm}\label{simple-affine}
Let $R$ be a finite commutative ring. Let $\mathcal{O}_1,\ldots,\mathcal{O}_m$ be the orbits of $U(R)$ on $\widehat{R}$ and fix $\chi_i \in \mathcal{O}_i$. Then a complete set of representatives of the isomorphism classes of simple $\mathbb{C}U(\Aff(R))$-modules is given by the modules
\[W_{(\mathcal{O}_i,\rho)} = \Ind^{U(\Aff(R))}_{R\rtimes \St_{U(R)}(\chi_i)}\chi_i\otimes \rho\]
with $\rho \in \widehat{\St_{U(R)}(\chi_i)}$.
\end{thm}

%%%%%%%%%%%%%%%%%%%%%%%%%%%%%%%%%%%%%%%%%%%%%%%%%%%%%%%%%%%%%%%%%%%%%%%%%%%%%%%%%%%%%%%%%%%%%%%%%%%%%%%%%%%%%%%%%%%%%%%%%%%%%%%%%%%%%%%%%%%%%%%%%%%%%%%%%%%%%%%%%%%%%%%%%%%%%%%%%%%%%%%%%%%%%%%%%%%%%%%%%%%%%%%%%%%%%%%%%%%%%%%%%%%%%%%%%%%%%%%%%%%%%%%%%%%%%%%%%%%%%%%%%%%%%%%%%%%%%%%%%%%%%%%%%%%%%%%%%%%%%%%%%%%%%%%%%%%%%%%%%%%%%%%%%%%%%%%%%%%%%%%%%%%%%%%%%%%%%%%%%%%%%%%%%%%%%%%%%%%%%%%%%%%%%%%%%%%%%%%%%%%%%%%%%%%%%%%%%%%%%%%%%%

\section{The quiver of the affine monoid of a finite commutative ring}
In this section we will compute the quiver of $\mathbb C\Aff(R)$ when $R$ is a finite commutative local ring.  The quiver $\mathbb C\Aff(R)$ for a finite commutative ring $R$ can then be computed using Propositions~\ref{p:tensor} and~\ref{p:direct.prod}.

So let $R$ be local with maximal ideal $\mathfrak m$.  Since $R$ is artinian, $\mathfrak m^n=0$ for some $n\geq 1$, and so  $R$ has no idempotents except $0$ and $1$.  Then the only idempotents in $\Aff(R)$ are the identity map $x$ and the constant mappings.  In particular, the idempotents form an $\mathscr R$-trivial submonoid and so we can apply Theorem~\ref{U-flat}.  Also note that the only regular $\mathscr J$-classes of $\Aff(R)$ are the group of units and the constant maps. Note that $\widetilde{L}_0=L_0$ is the set of constant maps and $\widetilde{L}_x=\Aff(R)\setminus L_0$ where we write $0$ for the constant map to $0$ and $x$ for the identity map $f(x)=x$.

We can use the constant map to $0$ and the identity map $x$ as our idempotent representatives of the regular $\mathscr J$-classes.
 Note that $G_0=\{0\}$ and $G_x =U(\Aff(R))$. In the language of Theorem~\ref{U-flat},  the trivial module corresponds to the unique simple module for $\mathbb CG_0$.  The simple modules obtained by extending the simple $\mathbb CU(\Aff(R))$-modules by having elements of $\Aff(R)\setminus U(\Aff(R))$ act as zero correspond to themselves in Theorem~\ref{U-flat}, viewed as irreducible representations of $G_x=U(\Aff(R))$.

%If the maximal ideal of $R$ is $(p)$ and $n$ is minimal with $p^n=0$, then the $\mathscr J$-classes of $\Aff(R)$ are $J_k=J_{p^kx}$ with $0\leq k\leq n$ and $J_k\leq_{\mathscr J} J_m$ if and only if $m\leq k$.
In summary, the regular $\mathscr J$-classes are $J_{x}=U(\Aff(R))=L_x=R_x=G_x$ and $J_{0}=L_0$, which consists of the constant mappings.  The $\widetilde{\mathscr L}$-classes are $\widetilde{L}_0=L_0=J_0$ and $\widetilde{L}_x = \Aff(R)\setminus L_0$.

\begin{prop}\label{p:no.end.bottom}
There are no arrows in $Q(\mathbb C\Aff(R))$ ending at the trivial module.	
\end{prop}
\begin{proof}
 If $f$ is the constant map to $0$ and $e$ is either the identity or the constant map to $0$, then $f\Aff(R)e=\{f\}$, from which it follows that $U^\flat=0$ in both cases; in the first case since $\widetilde{L}_x\cap f\Aff(R)=\emptyset$ and in the second since $\widetilde{L}_0\cap f\Aff(R)=\{f\} = f\Aff(R)\cap L_0$ and $[f]_{\equiv}-[f]_{\equiv}=0$.  Thus there are no arrows in the quiver ending at the trivial module.
\end{proof}

We first deal with the case that $R$ is a field, as this case is slightly different (and simpler).  This case follows from the results of~\cite[Section~6.2]{MarSte11}, which computes the quiver of a monoid obtained from a permutation group by adjoining the constant mappings, but we give an alternative proof using Theorem~\ref{U-flat}.

Let us recall some basic facts about characters of finite fields.  Let $F$ be a finite field of characteristic $p$ and $[F:\mathbb F_p]=n$ where $\mathbb F_p$ is the field of $p$ elements.  Then since $F$ has exponent $p$ as an abelian group, the characters $\chi\colon F\to U(\mathbb C)$ take values in the group $\mu_p$ of $p^{th}$-roots of unity, which being cyclic we can identify with the additive group of $\mathbb F_p$.  Moreover, any additive homomorphism between $\mathbb F_p$-vector spaces is automatically $\mathbb F_p$-linear.  Thus we can identify $\widehat{F}$ with $\Hom_{\mathbb F_p}(F,\mathbb F_p)$, which has $|F|$ elements; this correspondence sends the trivial character of $F$ to the zero map.  The corresponding group action of $U(F)$ is  given by $(u\alpha)(a) = \alpha(ua)$ for $u\in U(F)$, $\alpha\in   \Hom_{\mathbb F_p}(F,\mathbb F_p)$ and $a\in F$.
The trace map $\tau\colon F\to \mathbb F_p$ given by $\tau(a) = a+a^p+a^{p^2}+\cdots+a^{p^{n-1}}$ is a non-zero $\mathbb F_p$-linear map (as $F$ is a separable, even Galois, extension of $\mathbb F_p$).  Moreover, the orbit of $\tau$ under $U(F)$ consists of the $|F|-1$ non-zero elements of $\Hom_{\mathbb F_p}(F,\mathbb F_p)$; see~\cite[Theorem~2.24]{lidl}.  We conclude that the non-trivial characters in $\widehat{F}$ are in a single $U(F)$-orbit and the stabilizer in $U(F)$ of any non-trivial character of $F$ is trivial.  It follows from Theorem~\ref{simple-affine} that the irreducible representations of $U(\Aff(F))$ are the $|F|-1$ one-dimensional characters of $U(F)$ (inflated to $U(\Aff(F))$) and the $(|F|-1)$-dimensional simple module $\Ind_{F}^{U(\Aff(F))}\chi$, where $\chi$ is any non-trivial character of $F$ (the module is independent of the choice of $\chi$ as they all lie in the same orbit).

\begin{thm}\label{n=1}
Let $F$ be a finite field.  The vertices of $Q(\mathbb C\Aff(F))$ are the trivial module, the $|F|-1$ one-dimensional characters of $U(F)$ and an $(|F|-1)$-dimensional simple module $\Ind_{F}^{U(\Aff(F))}\chi$, where $\chi$ is a fixed non-trivial character of $F$.  There is an arrow from the trivial $\mathbb C\Aff(F)$-module to $\Ind_{F}^{U(\Aff(F))}\chi$.
\end{thm}
\begin{proof}
We can use the constant map to $0$ and the identity as our idempotent representatives of the regular $\mathscr J$-classes. Note that $\Aff(F)=U(\Aff(F))\cup L_0$. 	We know by Proposition~\ref{p:no.end.bottom} that no edges end at the trivial module.

There are also no arrows between $\mathbb CU(\Aff(F))$-modules because if $e$ is the identity, then $eI(e)I(e)e$ is the set of constant maps and the relation $\equiv$ on $L_e=U(\Aff(F))$ is the equality relation.  As no two elements of $U(\Aff(F))$ act as the same partial injection on the right of $R_e=U(\Aff(F))$, we deduce that $U^\flat=0$.  Thus the only possible arrows are from trivial module to a simple $\mathbb CU(\Aff(F))$-module.

Next, let $f$ be the identity and $e$ the constant map to $0$.  Note that $I(e)=\emptyset$ and $I(f)=L_e$.  It follows that  $\equiv$ is just the equality relation (since $fI(f)$ consists of constant maps and all elements of $L_e$ act as the same partial injection on $R_e=\{e\}$) and $X=L_e$ is the set of constant maps.  The action of $U(\Aff(F))$ on $X$ can be identified with its action on $F$ by mappings.  Then $U^\flat$ is the span of all differences $a-b$ with $a,b\in F$, which is the submodule of vectors in $\mathbb CF$ whose coefficients sum to $0$.  Since $U(\Aff(F))$ acts doubly transitively on $F$, we have that $U^\flat$ is a simple module of dimension $|F|-1$ by classical group representation theory~\cite{serrerep} and hence is $\Ind_{F}^{U(\Aff(F))}\chi$ where $\chi$ is a non-trivial character of $F$.
\end{proof}

Note that it follows from Theorem~\ref{n=1} that if $F$ is a field, then $\mathbb C\mathrm{Aff}(F)$ is a hereditary algebra as its quiver contains no paths of length $2$.  Therefore, if $R=F_1\times\cdots \times F_r$ is a direct product of fields (that is, $R$ is reduced or semisimple), then the basic algebra of $\mathbb C\Aff(R)$ is isomorphic to a tensor product of hereditary algebras and so its quiver presentation can be computed using~\cite{QuiverTensor}.

%%%%%%%%%%%%%%%%%%%%%%%%%%%%%%%%%%%%%%%%%%%%%%%%%%%%%%%%%%%%%%%%%%%%%%%%%%%%%%%%%%%%%%%%%%%%%%%%%%%%%%%%%%%%%%%%%%%%%%%%%%%%%%%%%%%%%%%%%%%%%%%%%%%%%%%%%%%%%%%%%%%%%%%%%%%%%%%%%%%%%%%%%%%%%%%%%%%%%%%%%%%%%%%%%%%%%%%%%%%%%%%%%%%%%%%%%%%%%%%%%%%%%%%%%%%%%%%%%%%%%%%%%%%%%%%%%%%%%%%%%%%%%%%%%%%%%%%%%%%%%%%%%%%%%%%%%%%%%%%%%%%%%%%%%%%%%%%%%%%%%%%%%%%%%%%%%%%%%%%%%%%%%%%%%%%%%%%%%%%%%%%%%%%%%%%%%%%%%%%%%%%%%%%%%%%%%%%%%%%%%%%%%%

Next let $R$ be a finite commutative local  ring which is not a field. Suppose that the maximal ideal is $\mathfrak m$ and note that $\mathfrak m\neq 0$ since $R$ is not a field.  The simple $\mathbb C\Aff(R)$-modules are the trivial module and the simple $\mathbb CU(\Aff(R))$-modules.

We already know that no edges of the quiver end at the trivial module from Proposition~\ref{p:no.end.bottom}.  We now characterize the edges emanating from the trivial module.  It turns out that there is exactly one.

 Note that $\widehat{R/\mathfrak m}$ embeds in $\widehat{R}$ via inflation of characters as precisely the characters $\chi$ with $\mathfrak m\subseteq \ker \chi$.  Moreover, if $\chi$ inflates $\theta$, then $u\chi$ inflates $(u+\mathfrak m)\theta$ for $u\in U(R)$.  Since $U(R)\to U(R/\mathfrak m)$ is surjective ($R$ being local) and $U(R/\mathfrak m)$ acts transitively on the non-trivial characters of $R/\mathfrak m$, we conclude that the characters $\chi$ with $\mathfrak m\subseteq \ker \chi\subsetneq R$ form a $U(R)$-orbit.
Note that, $\ker\chi=\mathfrak m$ if and only if $R/\mathfrak m$ is a prime field.

\begin{prop}\label{p:from.trivial}
There is exactly one edge of $Q(\mathbb C\Aff(R))$ with initial vertex the trivial module.  It ends at the vertex $W_{(\mathcal O,1)}$ where $\mathcal O$ is the orbit of a character $\chi$ of $R$ with $\mathfrak m\subseteq\ker \chi\subsetneq R$ and $1$ is the trivial representation of $\St_{U(R)}(\chi)$.
\end{prop}
\begin{proof}
We compute the module $U^{\flat}$ from Theorem~\ref{U-flat}.  Here $e$ is the constant map to $0$, $f$ is the identity, $I(e)=\emptyset$ and $fI(f)=I(f)=\Aff(R)\setminus U(\Aff(R))$. So the elements of $f\Aff(R)e$ not meeting $fI(f)I(e)e$ is just the set of constant mappings, which we identify with $R$.  Note that $\widetilde{L}_e=L_e$ is the set of constant mappings.

Observe that $R_e=\{e\}$ and all constant mappings act the same on $R_e$.  We claim $r\equiv s$ if and only if $r+\mathfrak m=s+\mathfrak m$.  First note that if $t,u\in R$ and $ax+b\in I(f)$, then $a\in \mathfrak m$ and so  $(at+b)-(au+b) = a(t-u)\in \mathfrak m$.  Thus $at+b+\mathfrak m=au+b+\mathfrak m$.  It follows that $\equiv$ is contained in the equivalence relation of congruence modulo $\mathfrak m$.

On the other hand, if $r+\mathfrak m=s+\mathfrak m$, then $r-s\in \mathfrak m$.  Then $(r-s)x+s\in fI(f)$ takes $0$ to $s$ and $1$ to $r$.  Thus $r\equiv s$.  We deduce that $\mathbb CX$ is  the inflation of the permutation module of $U(\Aff(R/\mathfrak m))$ acting on $R/\mathfrak m$ via the projection $\gamma\colon U(\Aff(R))\to U(\Aff(R/\mathfrak m))$ (which is surjective since $R$ is local and hence each element of $U(R/\mathfrak m)$ is the image of an element of $U(R)$).  Then $U^\flat$ is the subspace spanned by all differences of elements of $R/\mathfrak m$, which is the simple $U(\Aff(R/\mathfrak m))$-module $\Ind_{R/\mathfrak m}^{U(\Aff(R/\mathfrak m))}\theta$ where $\theta$ is a non-trivial character of $R/\mathfrak m$, as in the proof of Theorem~\ref{n=1}.  This inflates to the simple module $W_{(\mathcal O,1)}$ where $\mathcal O$ is the orbit of the inflation of $\theta$, i.e., of a character $\chi$ with kernel $\mathfrak m\subseteq \ker\chi\subsetneq R$.
\end{proof}

Put $G=U(\Aff(R))$.
Our next step is to compute the $\mathbb C[G\times G]$-module $U^\flat$ for the case $e=f$ is the identity.  In this case no two elements of $L_e=G=R_e$ act as the same partial injection on the right of $R_e$ and $I(e)=I(f)=fI(f)=\Aff(R)\setminus G$. Also $\widetilde{L}_e\setminus L_e=\{ax+b\mid a\in \mathfrak m\setminus \{0\}\}$. It follows that $U^\flat$ is spanned by the elements of $I(f)\setminus I(f)^2$, which is the  set of all mappings $ax+b$ with $a\in \mathfrak m\setminus \mathfrak m^2$.  Indeed, $I(f)^2$ is clearly the set of mappings $ax+b$ with $a\in \mathfrak m^2$ as it consists of all compositions $(cx+d)\circ (fx+g)$ with $c,f\in \mathfrak m$. Let $J=\{ax+b\mid a\in \mathfrak m\setminus \mathfrak m^2\}$.  Then $U^\flat = \mathbb CJ$.  Moreover, it is the permutation module associated to the action of $G\times G$ on $J$ given by $(ax+b,cx+d)\cdot (ex+f) = (ax+b)\circ (ex+f)\circ (cx+d)^{-1} = (ax+b)\circ (ex+f)\circ (c^{-1}x-c^{-1}d)$.  Note that $J$ is not usually a single $\mathscr J$-class, although it is when $\mathfrak m$ is a principal ideal.

Recall that two elements $a,b\in R$ are \emph{associates} if $a=ub$ with $u\in U(R)$.

\begin{thm}\label{t:quiver}
Let $R$ be a finite commutative local ring with maximal ideal $\mathfrak m\neq 0$. Let $p_1,\ldots, p_r$ consist of one element from each set of associates of elements of $\mathfrak m\setminus \mathfrak m^2$.  Let $\mathfrak a_i$ be the annihilator ideal of $p_i$, for $1\leq i\leq r$.  Then in $Q(\mathbb C\Aff(R))$:
\begin{enumerate}
\item  there is an arrow from the trivial module to $W_{(\mathcal O,1)}$ where $\mathcal O$ is the orbit of a character $\chi\colon R\to U(\mathbb C)$ with	$\mathfrak m\subseteq \ker \chi\subsetneq R$ and $1$ the trivial character of $\St_{U(R)}(\chi)$;
\item for each $1\leq i\leq r$, if $\mathcal O$ is an orbit of a character $\chi$ with $\mathfrak a_i\subseteq \ker \chi$ and $\rho$ is a character of $\St_{U(R)}(\chi)$ with $1+\mathfrak a_i\subseteq \ker \rho$, then there is one arrow from $W_{(\mathcal O,\rho)}$ to each simple module $W_{(\mathcal O',\rho')}$ such that $\mathcal O'$ is the orbit of a character $\chi'$ with $\chi'(p_ia)=\chi(a)$ for all $a\in R$ and $\rho|_{\St_{U(R)}(\chi')}=\rho'$, where we note that $\St_{U(R)}(\chi')\subseteq \St_{U(R)}(\chi)$ (note that, multiple edges can arise from different choices of $i$).
\end{enumerate}
\end{thm}
\begin{proof}
In light of Propositions~\ref{p:no.end.bottom} and~\ref{p:from.trivial}, we are left with computing arrows between simple $\mathbb CG$-modules.

Let $V,W$ be simple $\mathbb CG$-modules.  We need to find the multiplicity of $W\otimes_{\mathbb C} D(V)$ in $\mathbb CJ$.  Note that $\mathbb CJ$ is a permutation module for $G\times G$.  We claim that the orbits are the sets $\mathcal O_i=\{up_ix+b\mid u\in U(R), b\in R\}$. Indeed, if  $a\in\mathfrak m\setminus \mathfrak m^2$, then $a=up_i$ with $u\in U(R)$. Then $ax+b = (ux+b,x)\cdot p_ix$.  Conversely, if $(ax+b,cx+d)\in G\times G$, then $(ax+b)\circ p_ix\circ (cx+d)^{-1}= (ap_ix+b)\circ (c^{-1}x-c^{-1}d) = ac^{-1}p_ix-ac^{-1}p_id+b$.  Since $p_i,p_j$ are not associates if $i\neq j$, the claim follows.

Thus to find the multiplicity of $W\otimes_{\mathbb C}D(V)$ in $\mathbb CJ$, it suffices to compute its multiplicity in the permutation module $\mathbb C\mathcal O_i$ for each $i=1,\ldots, r$.  Let us compute the stabilizer $H$ of $p_ix$.  Note that $p_ix=(ax+b,cx+d)\cdot p_ix$ if and only if $p_ix=(ap_ix+b)\circ (c^{-1}x-c^{-1}d) = ac^{-1}p_ix-ac^{-1}p_id+b$.  Thus occurs if and only if $ac^{-1}p_i=p_i$ and $b=p_id$.  So $H = \{(ax+p_id, cx+d)\mid a+\mathfrak a_i = c+\mathfrak a_i\}$ since $\mathfrak a_i$ is the annihilator of $p_i$.
Therefore, $\mathbb C\mathcal O_i=\Ind_H^{G\times G}1_H$ where $1_H$ is the trivial $\mathbb CH$-module.

Suppose that $V=W_{(\mathcal O,\rho)}$ and $W=W_{(\mathcal O',\rho')}$ where $\mathcal O$ is the orbit of a character $\chi\in \widehat{R}$ and $\rho\in \widehat{\St_{U(R)}(\chi)}$ and $\mathcal O'$ is the orbit of a character $\chi'\in \widehat{R}$ with $\rho'\in \widehat{\St_{U(R)}(\chi')}$.  Note that $D(V)=W_{\overline{\mathcal O},\overline{\rho}}$ where we recall that if $\theta\in \widehat{A}$ for an abelian group $A$, then $\overline{\theta}(a) = \overline{\theta(a)}$ (the complex conjugate).  Also, since $\theta$ takes values in the unit circle, we have that $\overline{\theta(a)}=\theta(a^{-1})$.

Note that $G\times G\cong U(\Aff(R\times R))$ and $W\otimes_{\mathbb C}D(V)$ is the simple $\mathbb C[G\times G]$-module $\Ind^{G\times G}_{R^2\rtimes (\St_{U(R)}(\chi')\times \St_{U(R)}(\chi))} \chi'\overline{\chi}\otimes \rho'\overline{\rho}$ where $\chi'\overline{\chi}(a,b) =\chi'(a)\overline{\chi(b)}$ and $\rho'\overline{\rho}(c,d) = \rho'(c)\overline{\rho(d)}$.    For convenience put $K=R^2\rtimes (\St_{U(R)}(\chi')\times \St_{U(R)}(\chi))$.

So by Frobenius reciprocity, the multiplicity of $W\otimes_{\mathbb C}D(V)$ in $\mathbb C\mathcal O_i$ is the multiplicity of $1_H$ in $\Res^{G\times G}_H \Ind^{G\times G}_K\chi'\overline{\chi}\otimes \rho'\overline{\rho}$.  To compute this multiplicity, we make use of the Mackey decomposition.  Let $L=\{(a,c)\in U(R)^2\mid a+\mathfrak a_i=c+\mathfrak a_i\}$. Note that $K$ and \[HK=R^2\rtimes ((\St_{U(R)}(\chi')\times \St_{U(R)}(\chi))L)\] are normal subgroups of $G\times G$.  Thus we can identify $H\backslash (G\times G)/K$ with $(G\times G)/HK$.
Also, $H\cap K = \{(ax+p_id,cx+d)\mid (a,c)\in (\St_{U(R)}(\chi')\times \St_{U(R)}(\chi))\cap L\}$.

  Let $S$ be a set of coset representatives for \[U(R)/\St_{U(R)}(\chi')\St_{U(R)}(\chi)(1+\mathfrak a_i)\] containing $1$.  Then we claim that the elements of the form $(x,sx)$ with $s\in S$ form a set of coset representatives for $(G\times G)/HK$.

Since $K$ contains $R^2$, clearly $(ax+b,cx+d)HK = (ax,cx)HK$.  But $(a^{-1}x,a^{-1}x)\in H$ and so $(ax,cx)HK = (x,ca^{-1}x)HK$.  Now we can write $ca^{-1}=suvw$ with $s\in S$, $u\in \St_{U(R)}(\chi)$, $v\in \St_{U(R)}(\chi')$ and $w\in 1+\mathfrak a_i$.  Then $(vx,vwx)\in H$, $(v^{-1}x,ux)\in K$ and $(x,sx)(vx,vwx)(v^{-1}x,ux) = (x,svwux) = (x,ca^{-1}x)$. Thus $(ax+b,cx+d)HK=(x,sx)HK$.

Suppose that $(x,sx)HK=(x,s'x)HK$ with $s,s'\in S$.  Then $(x,s^{-1}s'x) = (ax+p_id,cx+d)(ex+f,gx+h)$ with $a+\mathfrak a_i=c+\mathfrak a_i$ and $e\in \St_{U(R)}(\chi')$ and $g\in \St_{U(R)}(\chi)$.  Note that we must have $ae=1$, $cg=s^{-1}s'$.  Note that $ce+\mathfrak a_i=ae+\mathfrak a_i=1+\mathfrak a_i$ and so $c\in \St_{U(R)}(\chi')(1+\mathfrak a_i)$.  Thus $s^{-1}s'\in \St_{U(R)}(\chi')\St_{U(R)}(\chi)(1+\mathfrak a_i)$ and so $s=s'$.

We now obtain from the Mackey decomposition, using that $K$ is normal,
\[\Res^{G\times G}_H \Ind_K^{G\times G} \chi'\overline{\chi}\otimes \rho'\overline{\rho} \cong \bigoplus_{s\in S} \Ind_{H\cap K}^H \Res_{H\cap K}^K\chi'(s\overline{\chi})\otimes \rho'\overline{\rho}.\] An application of Frobenius reciprocity shows that the multiplicity of $1_H$ in the right hand side of the above equation is the multiplicity of $1_{H\cap K}$ in $\bigoplus_{s\in S}\Res^K_{H\cap K}\chi'(s\overline{\chi})\otimes \rho'\overline{\rho}$.

Now if $(ax+p_id,cx+d)\in H\cap K$, then we have
\begin{equation}\label{eq:to.end.all}
(\chi'(s\overline{\chi})\otimes \rho'\overline{\rho})(ax+p_id,cx+d) = \chi'(p_id)\overline{\chi}(sd)\rho'(a)\overline{\rho(c)}.
\end{equation}
The right hand side of \eqref{eq:to.end.all} is $1$ for all elements of $H\cap K$ if and only if for all $d\in R$, we have $\chi'(p_id)=\chi(sd)=(s\chi)(d)$ and for all $a\in \St_{U(R)}(\chi')$ and $c\in \St_{U(R)}(\chi)$ with $a+\mathfrak a_i=c+\mathfrak a_i$ we have that $\rho'(a)=\rho(c)$.  Note that by definition of $S$, if $s\neq s'$, then $s\chi\neq s'\chi$ and so there is at most one $s\in S$ so that $\chi'(p_id)=(s\chi)(d)$ for all $d\in R$.  Replacing $\chi'$ by $s^{-1}\chi'$, which is in the same orbit as $\chi'$, we may assume that if the multiplicity of $1_{H\cap K}$ is non-zero, then $\chi'(p_id)=\chi(d)$ for all $d\in R$.  In particular, if $a\in \mathfrak a_i$, then  $\chi(a) =\chi'(p_ia)=\chi'(0)=1$ and so $\mathfrak a_i\subseteq \ker\chi$.  Also note that if $a\in \St_{U(R)}(\chi')$, then $(a\chi)(d) = \chi(ad)=\chi'(p_iad) = (a\chi')(p_id)=\chi'(p_id)=\chi(d)$, and so we have that $\St_{U(R)}(\chi')\subseteq \St_{U(R)}(\chi)$.

Next observe that if $u\in 1+\mathfrak a_i$, say $u=1+r$ with $r\in \mathfrak a_i$, then $(u\chi)(a) = \chi(ua)=\chi(a+ra)=\chi(a)$, as $\mathfrak a_i\subseteq \ker\chi$.  Thus $1+\mathfrak a_i\subseteq \St_{U(R)}(\chi)$ and $(x,ux)\in H\cap K$ and so $1=\rho'(1)=\rho(u)$.  Thus $1+\mathfrak a_i\subseteq \ker \rho$.  Also, if $a\in \St_{U(R)}(\chi')\subseteq \St_{U(R)}(\chi)$, then we have $(ax,ax)\in H\cap K$ and so $\rho'(a)=\rho(a)$, that is, $\rho'=\rho|_{\St_{U(R)}(\chi')}$.

Conversely, if $\mathfrak a_i\subseteq \ker \chi$, $1+\mathfrak a_i\subseteq \ker \rho$, $\chi'(p_id)=\chi(d)$ for all $d\in R$ and $\rho'=\rho|_{\St_{U(R)}(\chi')}$ (which makes sense since $\chi'(p_id)=\chi(d)$ implies $\St_{U(R)}(\chi')\subseteq \St_{U(R)}(\chi)$, as we saw above), then for $(ax+p_id,cx+d)\in H\cap K$, we just need to check that $\chi'(p_id)=\chi(d)$, which we are given, and that $\rho'(a)=\rho(c)$.  This latter follows since $1+\mathfrak a_i\subseteq \ker\rho$ and hence $\rho(c)=\rho(a)=\rho'(a)$, as $a+\mathfrak a_i=c+\mathfrak a_i$.

This completes the proof of the theorem.
\end{proof}

Retaining the above notation, note that there are $r$ loops based at the trivial module of $U(\Aff(R))$ because $\chi,\chi'$ and $\rho,\rho'$ all trivial meet the above condition for each $1\leq i\leq r$.

Let us specialize to the case of a principal ideal ring.

\begin{cor}\label{t:.prin.quiver}
Let $R$ be a finite commutative local principal ideal ring with maximal ideal $\mathfrak m=(p)$ with $p\neq 0$ and $n\geq 2$ minimal with  $\mathfrak m^n=0$.  Then in $Q(\mathbb C\Aff(R))$:
\begin{enumerate}
\item  there is an arrow from the trivial module to $W_{(\mathcal O,1)}$ where $\mathcal O$ is the orbit of a character $\chi\colon R\to U(\mathbb C)$ with	$(p)\subseteq \ker \chi\subsetneq R$ and $1$ the trivial character of $\St_{U(R)}(\chi)$;
\item if $\mathcal O$ is an orbit of a character $\chi$ with $(p^{n-1})\subset \ker \chi$ and $\rho$ is a character of $\St_{U(R)}(\chi)$ with $1+(p^{n-1})\subseteq \ker \rho$, then there is one arrow from $W_{(\mathcal O,\rho)}$ to each simple module $W_{(\mathcal O',\rho')}$ such that $\mathcal O'$ is the orbit of a character $\chi'$ with $\chi'(pa)=\chi(a)$ for all $a\in R$ and $\rho|_{\St_{U(R)}(\chi')}=\rho'$, where we note that $\St_{U(R)}(\chi')\subseteq \St_{U(R)}(\chi)$.
\end{enumerate}
In particular, there are no multiple edges in the quiver.
\end{cor}

Finite commutative local principal ideal rings are also called chain rings as these are precisely the finite commutative rings whose ideals form a chain.  Note that in the case $R=\mathbb Z_{p^n}$ with $p$ prime, the condition $(p)\subseteq \ker \chi\subsetneq R$ is equivalent to $(p)=\ker \chi$.

\begin{rmk}
Let $\mathcal O$ and $\mathcal O'$ be orbits of $U(R)$ on $\widehat{R}$.
Suppose that there is one arrow from a simple module $W_{(\mathcal O,\rho)}$ to a simple module $W_{(\mathcal O',\rho')}$.
If $\zeta\in \mathcal O$ and $\zeta'\in \mathcal O'$ then $\zeta$ and $\zeta'$ do not necessarily satisfy the conditions of Corollary~\ref{t:.prin.quiver}.(2). However there exist characters $\chi\in \mathcal O$ and $\chi'\in \mathcal O'$ such that they satisfy the conditions of the second part of corollary. To check that there is an arrow between two simple modules we have to carry out this procedure with all the characters of $R$.
\end{rmk}

The quiver of $\mathbb C\Aff(R)$ can now be computed for any finite commutative unital ring $R$ using Theorems~\ref{n=1} and~\ref{t:quiver} and Propositions~\ref{p:tensor} and~\ref{p:direct.prod}.

\begin{example}[Quiver of $\Aff(\mathbb{Z}_{4})$]\label{Z4}
The orbit representatives of $\mathbb{Z}_4^{\times}$ acting on $\widehat{\mathbb{Z}_4}$ and the characters of the group $\mathbb{Z}_4^{\times}$ are in the following tables, respectively:
\begin{center}
\begin{minipage}{1.5in}
\begin{tabular}{ l | ccc}
$\mathbb{Z}_4$   & $\chi_0$ & $\chi_1$ & $\chi_2$ \\ \hline
$0$              & $1$      & $1$      & $1$\\
$1$              & $1$      & $i$      & $-1$ \\
$2$              & $1$      & $-1$     & $1$ \\
$3$              & $1$      & $-i$     & $-1$
\end{tabular}
\end{minipage}
\begin{minipage}{1.5in}
\begin{tabular}{ l | c  rc }
$\mathbb{Z}_4^{\times}$   & $\rho_0$ & $\rho_1$ \\ \hline

$1$                       & $1$      & $1$ \\
$3$                       & $1$      & $-1$
\end{tabular}\\
\end{minipage}
\end{center}
By Corollary~\ref{t:.prin.quiver} the quiver of  the complex algebra of $\Aff(\mathbb{Z}_{4})$ is then as in Figure~\ref{fig:diagramsz4}.
\begin{figure}[htbp]
  \begin{center}
    \begin{tikzpicture}[>=latex, shorten >=0pt, shorten <=0pt, scale=0.9]
      \draw (1,4) node (1) {$\chi_0\otimes \rho_0$};
      \draw (4,4) node (2) {$\chi_2\otimes \rho_0$};
      \draw (7,3) node (3) {$\Ind_{\mathbb{Z}_{4}\rtimes \{1\}}^{\mathbb{Z}_{4}\rtimes \mathbb{Z}_{4}^{\times}}\chi_1$};
      \draw (1,2) node (4) {$\chi_0\otimes \rho_1$};
      \draw (4,2) node (5) {$\chi_2\otimes \rho_1$};
      \draw (3.5,0.5) node (6) {$\mathbb{C}$};
      \path (1) edge [loop left, left=10, thick]  node {} (1)
            (1) edge [->, bend left=10, thick] node {} (2)
            (2) edge [->, bend left=10, thick] node {} (3)
            (6) edge [->, bend left=40, thick] node {} (2);
    \end{tikzpicture}
  \end{center}
  \caption{Quiver of  $\Aff(\mathbb{Z}_{4})$
  \label{fig:diagramsz4}}
\end{figure}
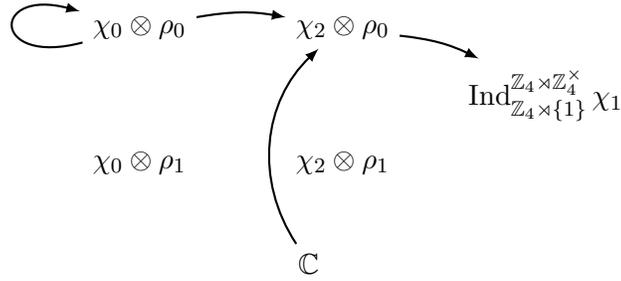
\end{example}

\begin{example}[Quiver of $\Aff(\mathbb{Z}_{8})$]\label{Z8}
The orbit representatives of $\mathbb{Z}_8^{\times}$ acting on $\widehat{\mathbb{Z}_8}$ are in the following table:
\begin{center}
\begin{tabular}{ l | cccc}
$\mathbb{Z}_8$   & $\chi_0$ & $\chi_1$       & $\chi_2$ & $\chi_3$ \\ \hline
$0$              & $1$      & $1$            & $1$      & $1$\\
$1$              & $1$      & $e^{\pi i/4}$  & $i$      & $-1$ \\
$2$              & $1$      & $i$            & $-1$     & $1$\\
$3$              & $1$      & $e^{3\pi i/4}$ & $-i$     & $-1$\\
$4$              & $1$      & $-1$           & $1$      & $1$\\
$5$              & $1$      & $e^{5\pi i/4}$ & $i$      & $-1$\\
$6$              & $1$      & $-i$           & $-1$     & $1$\\
$7$              & $1$      & $e^{7\pi i/4}$ & $-i$     & $-1$
\end{tabular}
\end{center}
We have $\St_{\mathbb{Z}_8^{\times}}(\chi_0)=\St_{\mathbb{Z}_8^{\times}}(\chi_3)=\mathbb{Z}_8^{\times}$, $\St_{\mathbb{Z}_8^{\times}}(\chi_1)=\{1\}$ and $\St_{\mathbb{Z}_8^{\times}}(\chi_2)=\{1,5\}$.
The characters of the group $\mathbb{Z}_8^{\times}$ and the subgroup $\{1,5\}$ are in the following tables:
\begin{center}
\begin{minipage}{2in}
\begin{tabular}{ l | cccc}
$\mathbb{Z}_8^{\times}$   & $\theta_0$ & $\theta_1$ & $\theta_2$ & $\theta_3$\\ \hline

$1$                       & $1$        & $1$        & $1$        & $1$       \\
$3$                       & $1$        & $-1$       & $1$        & $-1$      \\
$5$                       & $1$        & $1$        & $-1$       & $-1$      \\
$7$                       & $1$        & $-1$       & $-1$       & $1$
\end{tabular}
\end{minipage}
\begin{minipage}{2in}
\begin{tabular}{ l | c  rc }
$\{1,5\}$                 & $\alpha_0$ & $\alpha_1$ \\ \hline

$1$                       & $1$        & $1$ \\
$5$                       & $1$        & $-1$
\end{tabular}\\
\end{minipage}
\end{center}
By Corollary~\ref{t:.prin.quiver}, the quiver of the complex algebra of $\Aff(\mathbb{Z}_{8})$ is as in Figure~\ref{fig:diagramsZ8}.
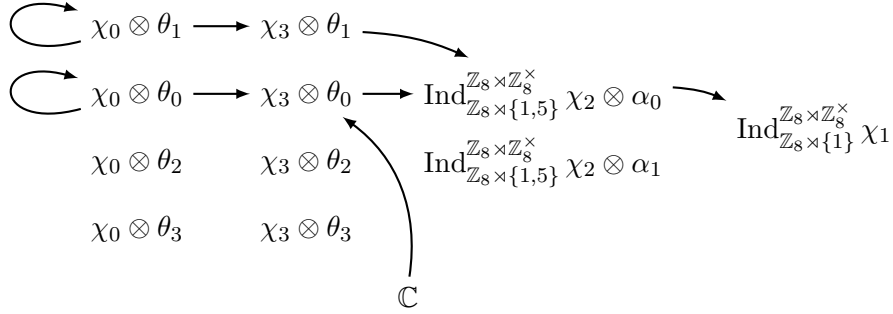
\begin{figure}[htbp]
  \begin{center}
    \begin{tikzpicture}[>=latex, shorten >=0pt, shorten <=0pt, scale=0.9]
      \draw (0,8) node (1) {$\chi_0\otimes \theta_1$};
      \draw (0,7) node (2) {$\chi_0\otimes \theta_0$};
      \draw (0,6) node (3) {$\chi_0\otimes \theta_2$};
      \draw (0,5) node (4) {$\chi_0\otimes \theta_3$};
      \draw (2.5,8) node (5) {$\chi_3\otimes \theta_1$};
      \draw (2.5,7) node (6) {$\chi_3\otimes \theta_0$};
      \draw (2.5,6) node (7) {$\chi_3\otimes \theta_2$};
      \draw (2.5,5) node (8) {$\chi_3\otimes \theta_3$};
      \draw (6,7) node (9) {$\Ind_{\mathbb{Z}_{8}\rtimes \{1,5\}}^{\mathbb{Z}_{8}\rtimes \mathbb{Z}_{8}^{\times}}\chi_2\otimes\alpha_0$};
      \draw (6,6) node (10) {$\Ind_{\mathbb{Z}_{8}\rtimes \{1,5\}}^{\mathbb{Z}_{8}\rtimes \mathbb{Z}_{8}^{\times}}\chi_2\otimes\alpha_1$};
      \draw (10,6.5) node (11) {$\Ind_{\mathbb{Z}_{8}\rtimes \{1\}}^{\mathbb{Z}_{8}\rtimes \mathbb{Z}_{8}^{\times}}\chi_1$};
      \draw (4,4) node (12) {$\mathbb{C}$};
      \path (1) edge [loop left, left=10, thick]  node {} (1)
            (2) edge [loop left, left=10, thick]  node {} (2)
            (1) edge [->, bend left=0, thick] node {} (5)
            (2) edge [->, bend left=0, thick] node {} (6)
            (5) edge [->, bend left=10, thick] node {} (9)
            (6) edge [->, bend left=0, thick] node {} (9)
            (9) edge [->, bend left=10, thick] node {} (11)
            (12) edge [->, bend right=31, thick] node {} (6);
    \end{tikzpicture}
  \end{center}
  \caption{Quiver of $\Aff(\mathbb{Z}_{8})$
  \label{fig:diagramsZ8}}
\end{figure}
\end{example}

\begin{example}[Quiver of $\Aff(\mathbb{Z}_{9})$]\label{Z9}
The orbit representatives of  $\mathbb{Z}_9^{\times}$ acting on $\widehat{\mathbb{Z}_9}$ are in the following table:
\begin{center}
\begin{tabular}{ l | ccc}
$\mathbb{Z}_9$   & $\chi_0$ & $\chi_1$       & $\chi_2$       \\ \hline
$0$              & $1$      & $1$             & $1$\\
$1$              & $1$      & $e^{2\pi i/9}$  & $e^{2\pi i/3}$\\
$2$              & $1$      & $e^{4\pi i/9}$  & $e^{4\pi i/3}$\\
$3$              & $1$      & $e^{2\pi i/3}$  & $1$           \\
$4$              & $1$      & $e^{8\pi i/9}$  & $e^{2\pi i/3}$\\
$5$              & $1$      & $e^{10\pi i/9}$ & $e^{4\pi i/3}$\\
$6$              & $1$      & $e^{4\pi i/3}$  & $1$           \\
$7$              & $1$      & $e^{14\pi i/9}$ & $e^{2\pi i/3}$\\
$8$              & $1$      & $e^{16\pi i/9}$ & $e^{4\pi i/3}$
\end{tabular}
\end{center}
We have $\St_{\mathbb{Z}_9^{\times}}(\chi_0)=\mathbb{Z}_9^{\times}$, $\St_{\mathbb{Z}_9^{\times}}(\chi_1)=\{1\}$ and $\St_{\mathbb{Z}_9^{\times}}(\chi_2)=\{1,4,7\}$.
The characters of the group $\mathbb{Z}_9^{\times}$ and the subgroup $\{1,4,7\}$ are in the following tables:
\begin{center}
\begin{minipage}{4in}
\begin{tabular}{ l | cccccc}
$\mathbb{Z}_9^{\times}$   & $\theta_0$ & $\theta_1$     & $\theta_2$     & $\theta_3$     & $\theta_4$     & $\theta_5$     \\ \hline

$1$                       & $1$        & $1$            & $1$            & $1$            & $1$            & $1$            \\
$2$                       & $1$        & $e^{\pi i/3}$  & $e^{2\pi i/3}$ & $-1$           & $e^{4\pi i/3}$ & $e^{5\pi i/3}$ \\
$4$                       & $1$        & $e^{2\pi i/3}$ & $e^{4\pi i/3}$ & $1$            & $e^{2\pi i/3}$ & $e^{4\pi i/3}$ \\
$5$                       & $1$        & $-e^{2\pi i/3}$& $e^{4\pi i/3}$ & $-1$           & $e^{2\pi i/3}$ & $-e^{4\pi i/3}$\\
$7$                       & $1$        & $-e^{\pi i/3}$ & $e^{2\pi i/3}$ & $1$            & $e^{4\pi i/3}$ & $-e^{5\pi i/3}$\\
$8$                       & $1$        & $-1$           & $1$            & $-1$           & $1$            & $-1$
\end{tabular}
\end{minipage}
\end{center}
\begin{center}
\begin{minipage}{2in}
\begin{tabular}{ l | c  rcc }
$\{1,4,7\}$               & $\alpha_0$ & $\alpha_1$    & $\alpha_2$     \\ \hline

$1$                       & $1$        & $1$           & $1$            \\
$4$                       & $1$        & $e^{2\pi i/3}$& $e^{4\pi i/3}$ \\
$7$                       & $1$        & $e^{4\pi i/3}$& $e^{2\pi i/3}$
\end{tabular}\\
\end{minipage}
\end{center}
By Corollary~\ref{t:.prin.quiver}, the quiver of the complex algebra of $\Aff(\mathbb{Z}_{9})$ is as in Figure~\ref{fig:diagrams}.
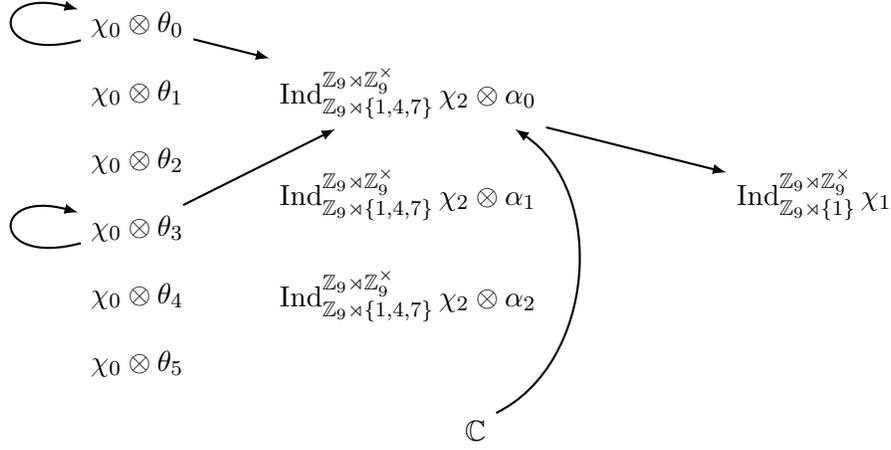
\begin{figure}[htbp]
  \begin{center}
    \begin{tikzpicture}[>=latex, shorten >=0pt, shorten <=0pt, scale=0.9]
      \draw (0,8) node (1) {$\chi_0\otimes \theta_0$};
      \draw (0,7) node (2) {$\chi_0\otimes \theta_1$};
      \draw (0,6) node (3) {$\chi_0\otimes \theta_2$};
      \draw (0,5) node (4) {$\chi_0\otimes \theta_3$};
      \draw (0,4) node (5) {$\chi_0\otimes \theta_4$};
      \draw (0,3) node (6) {$\chi_0\otimes \theta_5$};
      \draw (4,7) node (7) {$\Ind_{\mathbb{Z}_{9}\rtimes \{1,4,7\}}^{\mathbb{Z}_{9}\rtimes \mathbb{Z}_{9}^{\times}}\chi_2\otimes\alpha_0$};
      \draw (4,5.5) node (8) {$\Ind_{\mathbb{Z}_{9}\rtimes \{1,4,7\}}^{\mathbb{Z}_{9}\rtimes \mathbb{Z}_{9}^{\times}}\chi_2\otimes\alpha_1$};
      \draw (4,4) node (9) {$\Ind_{\mathbb{Z}_{9}\rtimes \{1,4,7\}}^{\mathbb{Z}_{9}\rtimes \mathbb{Z}_{9}^{\times}}\chi_2\otimes\alpha_2$};
      \draw (10,5.5) node (10) {$\Ind_{\mathbb{Z}_{9}\rtimes \{1\}}^{\mathbb{Z}_{9}\rtimes \mathbb{Z}_{9}^{\times}}\chi_1$};
      \draw (5,2) node (11) {$\mathbb{C}$};
      \path (1) edge [loop left, left=10, thick]  node {} (1)
            (4) edge [loop left, left=10, thick]  node {} (4)
            (1) edge [->, bend left=0, thick] node {} (7)
            (4) edge [->, bend left=0, thick] node {} (7)
            (7) edge [->, bend left=0, thick] node {} (10)
            (11) edge [->, bend right=60, thick] node {} (7);
    \end{tikzpicture}
  \end{center}
  \caption{Quiver of the affine monoid $\mathbb{Z}_{9}$
  \label{fig:diagrams}}
\end{figure}
\end{example}

\begin{example}[Quiver of $\Aff(\mathbb{Z}_6)$]\label{Z2Z3}
Note that $\mathbb Z_6\cong \mathbb Z_2\times \mathbb Z_3$.  So we can compute the quiver by applying Theorem~\ref{n=1} and Propositions~\ref{p:tensor} and~\ref{p:direct.prod}.
The orbit representatives of $\mathbb{Z}_{2}^{\times}$ acting on $\widehat{\mathbb{Z}_{2}}$, the orbit representatives of $\mathbb{Z}_{3}^{\times}$ acting on $\widehat{\mathbb{Z}_{3}}$ and the characters of the group $\mathbb{Z}_{3}^{\times}$ are in the following tables, respectively:
\begin{center}
\begin{minipage}{0.2\linewidth}
\begin{tabular}{ l | ccc}
$\mathbb{Z}_{2}$   & $\chi_0$ & $\chi_1$ \\ \hline
$0$                & $1$      & $1$\\
$1$                & $1$      & $-1$
\end{tabular}
\end{minipage}
\begin{minipage}{0.3\linewidth}
\begin{tabular}{ l | cc}
$\mathbb{Z}_{3}$   & $\chi'_0$ & $\chi'_1$ \\ \hline
$0$                & $1$        & $1$\\
$1$                & $1$        & $e^{2\pi i/3}$\\
$2$                & $1$        & $e^{4\pi i/3}$
\end{tabular}
\end{minipage}
\begin{minipage}{0.2\linewidth}
\begin{tabular}{ l | c  rc }
$\mathbb{Z}_{3}^{\times}$   & $\rho_0$  & $\rho_1$ \\ \hline

$1$                         & $1$         & $1$ \\
$2$                         & $1$         & $-1$
\end{tabular}
\end{minipage}
\end{center}
By Theorem~\ref{n=1} the quivers of the complex algebra of $\Aff(\mathbb{Z}_{2})$ and $\Aff(\mathbb{Z}_{3})$ are as in Figure~\ref{fig:diagramsz2z3}, respectively.
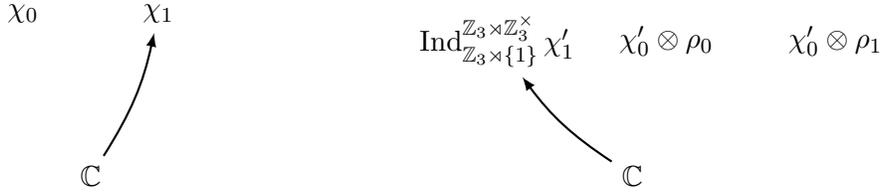
\begin{figure}[ht]
  \begin{center}
    \begin{tikzpicture}[>=latex, shorten >=0pt, shorten <=0pt, scale=0.9]
      \draw (-3.5,.4)   node (1) {$\chi_0$};
      \draw (-1.5,.4)   node (2) {$\chi_1$};

      \draw (6,0)  node (3) {$\chi'_0\otimes \rho_0$};
      \draw (8.5,0) node (4) {$\chi'_0\otimes \rho_1$};
      \draw (3.5,0)   node (5) {$\Ind_{\mathbb{Z}_3\rtimes \{1\}}^{\mathbb{Z}_3\rtimes \mathbb{Z}_3^{\times}}\chi'_1$};

      \draw (5.5,-2)   node (6) {$\mathbb{C}$};
      \draw (-2.5,-2)   node (7) {$\mathbb{C}$};

      \path
             (7) edge [->, bend right=10, thick] node {} (2)
             (6) edge [->, bend left=10, thick] node {} (5);
    \end{tikzpicture}
  \end{center}
  \caption{Quivers of $\Aff(\mathbb{Z}_{2})$ and $\Aff(\mathbb{Z}_{3})$}
  \label{fig:diagramsz2z3}
\end{figure}

Now, by Proposition~\ref{p:tensor} the quiver of the complex algebra of $\Aff(\mathbb{Z}_6)$ is as in Figure~\ref{fig:diagramsz2-z3}, where we have omitted the labelling of the vertices.
\begin{figure}[ht]
  \begin{center}
    \begin{tikzpicture}[>=latex, shorten >=0pt, shorten <=0pt, scale=0.9]
      \draw (9,3.5)  node[circle,fill,inner sep=2pt] (1)  {};
      \draw (11,3.5)   node[circle,fill,inner sep=2pt] (2)  {};
      \draw (5,3.5)   node[circle,fill,inner sep=2pt] (3)  {};
      \draw (7,3.5)   node[circle,fill,inner sep=2pt] (4)  {};
      \draw (-2,2.3)     node[circle,fill,inner sep=2pt] (5)  {};
      \draw (2.5,2.3)   node[circle,fill,inner sep=2pt] (6)  {};

      \draw (-1,1)    node[circle,fill,inner sep=2pt] (7)  {};
      \draw (1,1)     node[circle,fill,inner sep=2pt] (8)  {};

      \draw (7,0)    node[circle,fill,inner sep=2pt] (9)  {};
      \draw (9,0)   node[circle,fill,inner sep=2pt] (10) {};
      \draw (5,0)   node[circle,fill,inner sep=2pt] (11) {};

      \draw (4,-2)   node[circle,fill,inner sep=2pt] (12) {};

      \path  (7)  edge [->, bend left=10, thick] node {} (5)
             (8)  edge [->, bend right=10, thick]node {} (6)
             (9)  edge [->, bend right=10, thick] node {} (1)
             (10) edge [->, bend right=10, thick] node {} (2)
             (11) edge [->, bend left=10, thick] node {} (6)
             (12) edge [->, bend right=10, thick] node {} (8)
             (12) edge [->, bend left=10, thick] node {} (11);
    \end{tikzpicture}
  \end{center}
  \caption{Quiver of $\Aff(\mathbb{Z}_{6})$}
  \label{fig:diagramsz2-z3}
\end{figure}
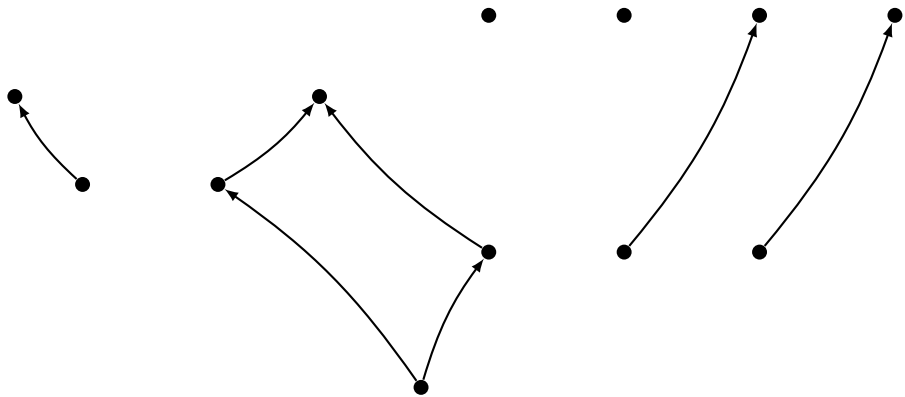
\end{example}

%%%%%%%%%%%%%%%%%%%%%%%%%%%%%%%%%%%%%%%%%%%%%%%%%%%%%%%%%%%%%%%%%%%%%%%%%%%%%%%%%%%%%%%%%%%%%%%%%%%%%%%%%%%%%%%%%%%%%%%%%%%%%%%%%%%%%%%%%%%%%%%%%%%%%%%%%%%%%%%%%%%%%%%%%%%%%%%%%%%%%%%%%%%%%%%%%%%%%%%%%%%%%%%%%%%%%%%%%%%%%%%%%%%%%%%%%%%%%%%%%%%%%%%%%%%%%%%%%%%%%%%%%%%%%%%%%%%%%%%%%%%%%%%%%%%%%%%%%%%%%%%%%%%%%%%%%%%%%%%%%%%%%%%%%%%%%%%%%%%%%%%%%%%%%%%%%%%%%%%%%%%%%%%%%%%%%%%%%%%%%%%%%%%%%%%%%%%%%%%%%%%%%%%%%%%%%%%%%%%%%%%%%%

\section*{Acknowledgments}
The first author was partially supported by CMUP, which is financed by
national funds through FCT -- Fundação para a Ciência e a Tecnologia,
I.P., under the project UIDB/00144/2020. The first author also
acknowledges FCT support through a contract based on the “Lei do
Emprego Científico” (DL 57/2016).
This work was performed while the first author was visiting the City College of New York. He thanks the College for its warm hospitality.  The second author was supported in part by a PSC-CUNY grant. We are grateful to the anonymous referee for their careful reading of the paper.

%\bibliographystyle{abbrv}
%\bibliography{ref-Rep-Qui}

\end{document}